\documentclass[11pt]{article}
\usepackage{fullpage}
\usepackage{amsmath,amssymb,amsthm}
\usepackage[pdfborder={0 0 0},colorlinks=true,linkcolor=blue,citecolor=blue,urlcolor=blue]{hyperref}
\usepackage[
backend=biber,
style=alphabetic,
sorting=nyt,
maxnames=99,
doi=false,
url=false,
isbn=false
]{biblatex}
\usepackage{caption}
\usepackage{cleveref}
\usepackage{comment}
\usepackage{tikz}
\usetikzlibrary{shapes.misc}

\newtheorem{lemma}{Lemma}[section]

\newtheorem{theorem}[lemma]{Theorem}
\theoremstyle{definition}

\newtheorem{example}[lemma]{Example}

\newcommand{\ZZ}{\mathbb{Z}}
\newcommand{\Exp}{\operatorname{Exp}}

\title{Microscopic sorting in a multispecies exclusion process}
\author{
Balázs Maga \and Bálint Tóth
}
\date{}

\begin{document}

\maketitle

\begin{abstract}
    We study a one-dimensional multispecies TASEP model with each particle type having an associated speed parameter and faster particles overtaking slower ones with the rate being the positive part of their speed difference. After expressing the sorting time of a finite system given a reverse-speed initial condition as the maximum of sorting times for all different fast-slow cuts, we prove that forbidding fast-fast and slow-slow overtakes does not alter the law of the sorting time. This allows for an LPP formulation and an analytic expression of the sorting time.

\end{abstract}

\begingroup
\renewcommand{\thefootnote}{}
\footnotetext{%
\hspace{-7mm} HUN-REN Alfréd Rényi Institute of Mathematics,
Reáltanoda utca 13--15, 1053 Budapest, Hungary.\\
E-mail: \texttt{magab@renyi.hu}, \texttt{toth.balint@renyi.hu}. \\
Supported by the
NKFIH Grant no. 152922, funded by the Ministry of Innovation
and Technology of Hungary from the National Research, Development and
Innovation Fund.
}
\addtocounter{footnote}{-1}
\endgroup

\section{Introduction}

Last-passage percolation (LPP) is closely connected to a number of interacting particle systems and stochastic sorting models. A classical example is the correspondence between the totally asymmetric simple exclusion process (TASEP) and exponential LPP, going back to Rost \cite{Rost1981}; this connection has played a central role in the exact analysis of TASEP and related growth models, notably in Johansson's work \cite{Johansson2000}. More recently, connections between LPP and finite particle systems with sorting dynamics have also appeared, in particular for the oriented swap process \cite{BisiCundenGibbonsRomik2022}. 

In this paper we study a multispecies TASEP, akin to the one studied in \cite{Cantini2016, KimWilliams2023}. We arrange $N$ particles along an interval each of them having a type in $\{0,\ldots,K-1\}$, with associated speed parameter
\[
    v_0<v_1<\cdots<v_{K-1}.
\]
Throughout $K\geq 2$ is fixed. Whenever neighboring particles have types $i$ and $j$ in this order, they
interchange at rate $v_i-v_j$ if $i>j$.  Interchanges in the other direction
are not allowed. Drawing motivation from the studies concerning the ordered swap process \cite{AngelHolroydRomik2009} and random sorting networks \cite{AngelHolroydRomikVirag2007}, we focus on the balanced reverse-speed initial configuration
\[
    (K-1)^n(K-2)^n\dots 0^n,
\]
where $n=N/K$. (We tacitly assume $K|N$.)
We are interested in the sorting time $T_N$ as $N\to\infty$, the time at which the process reaches the completely sorted configuration $0^n1^n\dots(K-1)^n$.

Our motivation for considering this rate structure rather than that of the oriented swap process, is that as established by \cite{TothValko2003} (considering interacting particle systems in substantially higher generality), this Markov process admits a hydrodynamic limit, being valid even beyond shock formation for $K=3$ \cite{FritzToth2004}, in which case the limit is the Leroux system. For general $K$, the limiting system of conservation laws was analyzed from a PDE point of view in \cite{schutz2026hydrodynamicsboundaryinducedphasetransitions}. Investigating the sorting problem from this perspective is left for future work.

\subsection{The connection to LPP}

Let us consider the $K=2$ case as a simple motivating example. Label the fast particles from
the contact point outwards by $1,\ldots,n$, and do the same for the slow
particles.  If $L^{(a,b)}$ is the time at which fast particle $a$ overtakes
slow particle $b$, then
\[
    L^{(a,b)}=\max\{L^{(a-1,b)},L^{(a,b-1)}\}+\tau_{(a,b)},
    \qquad \tau_{(a,b)}\sim\Exp(x_a+y_b)\ \text{are independent},
\]
with boundary values $L^{(a, 0)}=L^{(0, b)}=0$ for every positive integer $a, b$.  Thus $L^{(n,n)}$ is a
two-dimensional last-passage time, in particular, the usual LPP shape theorem shows $\frac{T_N}{N}\xrightarrow{a.s.} \frac{2}{v_1-v_0}$.

The previous formulation as an LPP problem does not directly generalize for $K\geq 3$. 
Still one is able to arrive at a similar problem. Fix $m\in\{1,\ldots,K-1\}$. The \emph{$m$-slow} particles are those of types
$0,\ldots,m-1$, and the remaining particles are called \emph{$m$-fast}. (As within our arguments, $m$ is mostly fixed, we often simply use the terms \emph{slow} and \emph{fast}.) Let $T_{N,m}$
be the time at which every fast particle has passed every slow particle. We refer to this event as $m$-sorting.
Then
\begin{equation}\label{eq:total-time}
    T_N=\max_{1\leq m<K}T_{N,m}.
\end{equation}
The first naive idea now is to proceed as above: label both the slow and the fast particles outwards from the contact point and
hence turn the problem until complete sorting into an LPP problem. However, there are two difficulties to handle:
\begin{itemize}
    \item Fast-fast and slow-slow overtakes can happen, thus $m$-sorting actually cannot be cast
    as an LPP problem.
    \item Even if we rule out fast-fast and slow-slow overtakes, we end up with an inhomogeneous LPP instead of a homogeneous one.
\end{itemize}

The second issue is more of a technical inconvenience: while inhomogeneity makes matters more complicated, results about the passage time asymptotics of 
inhomogeneous LPPs are available. The first issue is more fundamental. Overtakes among particles on the same side of the cut continually rearrange the parameters encountered by subsequent fast--slow crossings, so the direct LPP representation above breaks down. In contrast with the homogeneous oriented swap process, the binary fast--slow projection is not itself Markovian: these internal rearrangements change the identities, and hence the interchange rates, of particles subsequently meeting across a fast--slow interface.

\subsection{The main results}

Our first theorem asserts the somewhat surprising result
that fast-fast and slow-slow overtakes actually do not affect the distribution of the sorting time. To make this more explicit, once $m$ is fixed we define the \emph{suppressed model} by forbidding slow (resp. fast) particles overtaking other slow (resp. fast) particles. This suppression freezes the order of both slow and fast particles, implying that starting from the reverse speed initial condition, there is a unique attainable absorbing state, in which the configuration is $m$-sorted. We denote the corresponding sorting time by $T_{N, m}^*$.

\begin{theorem} \label{thm:fast-fast_slow-slow_makes_no_difference}
    With the notation above, $T_{N, m}^*\stackrel{d}{=} T_{N, m}$. 
    
    More generally, the same distributional identity stands for arbitrary overtake rates
    between pairs of fast and pairs of slow particles.
\end{theorem}

This invariance relies on a recent result on inhomogeneous LPPs \cite{BatesEmrahMartinSeppalainenSorensen2026}, establishing permutation invariance for LPP parameters.

Using \Cref{thm:fast-fast_slow-slow_makes_no_difference}, we can cast the study of $T_{N, m}$ as an LPP problem and we prove the following:

\begin{theorem} \label{thm:sorting_time}
    With the notation above, considering the independent coupling for $N=1, 2, \dots$,
    \begin{equation}\label{eq:main1}
    \frac{T_{N, m}}{N}\xrightarrow{a.s.}\frac{1}{K}\left(\sum_{i=0}^{m-1}\frac{1}{z-v_i}+\sum_{i=m}^{K-1}\frac{1}{v_i-z}\right) \quad m=1,2, \dots, K-1,
    \end{equation}
    and hence taking maximum,
    \begin{equation}\label{eq:main2}
    \frac{T_N}{N}\xrightarrow{a.s.}
    \frac{1}{K}\max_{1\leq m<K}\inf_{v_{m-1}<z<v_{m}}
    \left(
        \sum_{i=0}^{m-1}\frac{1}{z-v_i}+\sum_{i=m}^{K-1}\frac{1}{v_i-z}
    \right).
    \end{equation}
\end{theorem}

Observe that if \eqref{eq:main2} has a unique maximizing $m=m_0$, then with probability $1-o(1)$ the last overtake is a particle of type $m_0$ overtaking a particle of type $m_0-1$. Indeed, only such overtakes are prerequisites of $m_0$-sorting but no other $m$-sorting, thus they constitute the only possible source of the asymptotically almost sure inequality $T_{N, m_0}>\max_{m\neq m_0}T_{N, m}$.

\begin{example}
    For even $K$, let the set of speeds be $\{i/K\}_{i=1}^{K}$. Then the infimum is maximal for $m=K/2$ in which case it is attained at $z=\frac{K-1}{2}$, and we obtain
    $$\frac{T_N}{N}\to 2\sum_{i=0}^{K/2-1}\frac{1}{i+1/2}\sim_{K\to\infty} 2\log(K/2).$$
\end{example}
\begin{example}
    For even $K$ and $\beta>1$, let the set of speeds be $\{\pm \beta^i\}_{i=0}^{K/2-1}$. The infimum is maximal again for $m=K/2$, in which case it is attained at $z=0$, and we obtain
    $$\frac{T_N}{N}\to 2\sum_{i=0}^{K/2-1}\beta^{-i}\to_{K\to\infty}2\left(1-\frac{1}{\beta}\right).$$
\end{example}

\section{The suppressed model}

First we need to understand better the suppressed model. We work in the slightly more general setting where individual particles may have distinct speeds: in this case $m$ is simply the number of the slow particles while $N-m$ is the number of the fast particles. Denote by $u_1, \dots, u_{N-m}$ the speed of the fast particles, by $w_1, \dots w_m$ the speed of the slow ones. Set
\[
x_a=u_a\quad (a=1,2, \dots, N-m),\qquad y_b=-w_b\quad (b=1, 2, \dots, m).
\]
Then the interchange rate of fast particle $a$ and slow particle $b$ is
\[
x_a+y_b=u_a-w_b.
\]

In the suppressed model, the initial configuration uniquely determines the absorbing state the model ends up in. We say that a configuration is {\it admissible} if it can be attained from the reverse-speed initial configuration by performing a number of permitted overtakes.
The absorption time from an
admissible initial configuration $\sigma$ is denoted by $T_\sigma^*$.

In the reverse-speed configuration, let $L^{(a,b)}$ be the time at which fast particle
$a$ overtakes slow particle $b$.  These particles become neighbors after all
required overtakes indexed by $(a-1,b)$ and $(a,b-1)$ have taken place.
Consequently,
\begin{equation}\label{eq:lpp-recursion}
    L^ {(a,b)}=\max\{L^ {(a-1,b)},L^{(a,b-1)}\}+\tau_{(a,b)},
    \qquad \tau_{(a,b)}\sim\Exp(x_a+y_b)\ \text{are independent},
\end{equation}  
with boundary values $L^{(a, 0)}=L^{(0, b)}=0$ for every positive integer $a, b$.
Equivalently,
\begin{equation}\label{eq:lpp}
    L^{(a,b)}=\max_{\Gamma:(1,1)\nearrow(a,b)}
    \sum_{u\in\Gamma}\tau_{u}.
\end{equation}
Here the maximum is over up-right lattice paths. We slightly extend this notation:
for a pair of distinct points $p=(p_1, p_2),q=(q_1, q_2)\in \ZZ_{>0}^2$ with $p_1\leq q_1, p_2\leq q_2$, the last passage time $L_p^q$ between $p$ and $q$ is
\begin{equation}\label{eq:lpp2}
    L_p^q=\max_{\Gamma:p\nearrow q}
    \sum_{u\in\Gamma}\tau_{u}.
\end{equation}
so in particular,
$$L^ {(a, b)}=L_{(1,1)}^{(a,b)}.$$
That is, in the reverse-speed configuration the absorption time we are interested in is $L^{(N-m, m)}$.

For a general admissible initial configuration $\sigma$, some of the required
overtakes have already occurred. The vertices corresponding to these commited overtakes form a (bottom-up) Ferrers diagram, i.e., the dot version of a Young diagram (\Cref{fig:ferrers}). (\Cref{fig:ferrers}). This region is bounded by a down-right lattice curve $\gamma=\gamma_{\sigma}$ connecting $(0, m)$ to $(N-m, 0)$, which can be encoded by a \emph{coding string} $\{D, R\}^N$, telling that starting to draw the curve from $(0, m)$, which steps go to the right and which ones go downwards. In particular, for the reverse-speed configuration, the corresponding curve consists of the segments $[(0, m), (0,0)]$ and $[(0, 0), (N-m, 0)]$, and the corresponding string is $D^{m}R^{N-m}$. We say that $u$ is a \emph{$DR$-corner of $\gamma$} if the preceding entry of the coding string $D$, while the subsequent one is $R$.
Finally, let $\gamma^{+}=\gamma_\sigma^+$ denote the set of lattice vertices immediately northeast of the $DR$-corners of $\gamma_\sigma$. These correspond to the possible "first overtakes" that can happen once we start the dynamics from $\sigma$. In this case, we have
\begin{equation}\label{eq:boundary-lpp}
    T_\sigma^*=\max_{u\in \gamma_\sigma^+}L_u^{(N-m, m)}
\end{equation}

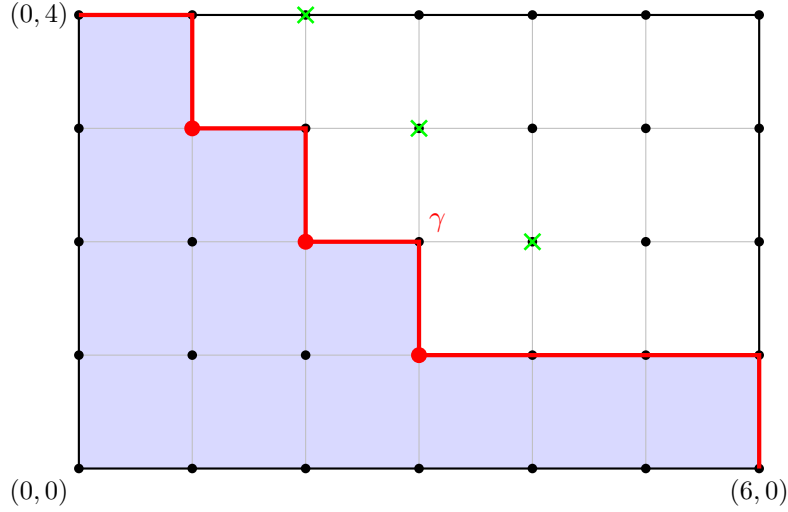
\begin{figure}[ht]
    \centering
    \begin{tikzpicture}[scale=1.5]
    \def\r{6} 
    \def\s{4} 

    \fill[blue!15]
        (0,0) --
        (0,4) --
        (1,4) --
        (1,3) --
        (2,3) --
        (2,2) --
        (3,2) --
        (3,1) --
        (4,1) --
        (6,1) --
        (6,0) --
        cycle;

    \draw[gray!50, step=1] (0,0) grid (\r,\s);

    \draw[thick] (0,0) rectangle (\r,\s);

    \foreach \x in {0,...,\r}
        \foreach \y in {0,...,\s}
            \fill (\x,\y) circle (1.2pt);

    %
    \draw[ultra thick, red]
        (0,4) --
        (1,4) --
        (1,3) --
        (2,3) --
        (2,2) --
        (3,2) --
        (3,1) --
        (4,1) --
        (5,1) --
        (6,1) --
        (6,0);

    \node[red, above right] at (3,2) {$\gamma$};

    \fill[red] (1,3) circle (2pt);
    \fill[red] (2,2) circle (2pt);
    \fill[red] (3,1) circle (2pt);

    \node[cross out, draw=green, line width=1pt, minimum size=5pt, inner sep=0pt] at (2,4) {};
\node[cross out, draw=green, line width=1pt, minimum size=5pt, inner sep=0pt] at (3,3) {};
\node[cross out, draw=green, line width=1pt, minimum size=5pt, inner sep=0pt] at (4,2) {};

    \node[below left] at (0,0) {\small $(0,0)$};
    \node[below] at (\r,0) {\small $(6,0)$};
    \node[left] at (0,\s) {\small $(0,4)$};

\end{tikzpicture}
    \caption{\small A Ferrers diagram in the case of $N=10$ particles of which $m=4$ is slow, bounded by the red down-right curve $\gamma$. The vertices of the positive quadrant covered by the shaded area show that starting from the reverse-speed configuration, which overtakes have already happened: the number of slow particles overtook by the $i$th fast particle (counted from the contact point) is the number of vertices of the form $(i, j)$ not above $\gamma$. Hence this configuration corresponds to the fast ($F$) and slow ($S$) particles ordered as $SFFFSFSFSF$  The $R-D$ string encoding $\gamma$ is $RDRDRDRRRD$, obtained by reflecting the $S-F$ string and replacing the symbols as $F\mapsto R, S\mapsto D$. The $DR$-corners of $\gamma$ are indicated by red dots, while the sites in $\gamma^{+}$ -- corresponding the next possible overtakes -- are indicated by green crosses.} 
    \label{fig:ferrers}
\end{figure}

We call consecutive columns of the Ferrers diagram a column block if they
have the same height, and define row blocks analogously. These correspond to contiguous blocks of fast/slow particles in $\sigma$.

Our first important lemma towards proving \Cref{thm:fast-fast_slow-slow_makes_no_difference} is a simple consequence of \cite[Theorem~3.1]{BatesEmrahMartinSeppalainenSorensen2026}

\begin{lemma}[LPP invariance]\label{lem:exponential-block}
Consider an LPP in the vertices of the positive quadrant with parameters
$\tau_{(a,b)}\sim\Exp(x_a+y_b)$, where $x_a+y_b>0$ for every $a,b$.  The law of $T_\sigma^*$ is
invariant under permutations of the $x$-parameters in a column block and of
the $y$-parameters in a row block.
\end{lemma}

\begin{proof}
    \cite[Theorem~3.1]{BatesEmrahMartinSeppalainenSorensen2026} tells us that in the setup of the Lemma, permuting column parameters in an interval $[a, a']$ of columns does not change the joint law of passage times $\{L_p^q:p_1, q_1\notin (a, a')\}$. However, inspecting \eqref{eq:boundary-lpp} this implies the invariance of $T_\sigma^*$ for such permutations of the parameters.
\end{proof}

To conclude this section, we make the following simple observation, for
which we extend the notation $T_\sigma^*$ to arbitrary, possibly
non-admissible configurations $\sigma$: it denotes the absorption time of
the suppressed model started from $\sigma$.

\begin{lemma}[Frozen sorting and particle blocks]
\label{lemma:in_block_exchange_invariance}
In the suppressed model, the law of $T_\sigma^*$ is invariant under
permuting labels inside any block of fast particles or inside any block of
slow particles.
\end{lemma}

\begin{proof}
For an arbitrary configuration $\sigma$, its fast--slow projection
determines the corresponding Young diagram, while the order of the fast
and slow labels determines the assignment of the $x$-parameters to its
columns and of the $y$-parameters to its rows. Thus permuting labels inside
a particle block amounts precisely to permuting the corresponding
$x$-parameters in a column block or the $y$-parameters in a row block.
The claim follows from Lemma~\ref{lem:exponential-block}.
\end{proof}

\section{Removing suppression}

First we allow deterministic fast-fast and slow-slow overtakes. Let $\Sigma=\{(t_1,i_1, j_1), \dots, (t_s,i_s, j_s)\}$ be a scenario of such overtakes: $t_1<t_2<\dots<t_q$, $i_\alpha, j_\alpha$ are labels of a pair of fast (resp. slow) particles, and at time $t_\alpha$, $i_\alpha$ and $j_\alpha$ are interchanged if they are neighbors at the time. The corresponding sorting time is denoted by $T_{\sigma}^{*,\Sigma}$. To formalize the proof of the following lemma, we note that after fixing an arbitrary reference state, i.e., particle configuration,  we can encode all states with elements of $S_N$, i.e., permutations of $\{1, 2, \dots, N\}$. 

\begin{lemma} \label{lemma:deterministic_scenario_independence}
    $T_{\sigma}^{*,\Sigma}\stackrel{d}{=} T_{\sigma}^*$ for any finite $\Sigma$ and initial condition $\sigma\in S_N$.
\end{lemma}

\begin{proof}
    We do induction on $|\Sigma|$, for $|\Sigma|=0$ the statement is tautological. 
    
    Consider now $|\Sigma|=s$ above. Let $\sigma(t_1)\in S_N$ be the state at time $t_1$ right before the (potential) interchange at time $t_1$, and for $\sigma'\in S_N$, denote by $(a_1\leftrightarrow b_1)\circ \sigma'$ the state obtained by interchanging $a_1, b_1$, if they are adjacent.
    \begin{align*}
    T_{\sigma}^{*,\Sigma}&=\mathbf{1}_{T_{\sigma}^{*,\Sigma}<t_1} T_{\sigma}^{*,\Sigma}+\mathbf{1}_{T_{\sigma}^{*,\Sigma}>t_1}\sum_{\sigma'\in S_N}\mathbf{1}_{\sigma(t_1)=\sigma'}\left(t_1+T_{(a_1\leftrightarrow b_1) \circ \sigma'}^{*,\Sigma-(t_1, a_1, b_1)}\right) \\
    &\stackrel{d}{=}\mathbf{1}_{T_{\sigma}^*<t_1} T_{\sigma}^*+\mathbf{1}_{T_{\sigma}^*>t_1}\sum_{\sigma'\in S_N}\mathbf{1}_{\sigma(t_1)=\sigma'}\left(t_1+T_{(a_1\leftrightarrow b_1)^ \circ \sigma'}^*\right)
    \\ 
    &\stackrel{d}{=} \mathbf{1}_{T_{\sigma}^*<t_1} T_{\sigma}^*+\mathbf{1}_{T_{\sigma}^*>t_1}\sum_{\sigma'\in S_N}\mathbf{1}_{\sigma(t_1)=\sigma'}\left(t_1+T_{ \sigma'}^*\right) \\
    &\stackrel{d}{=} T_{\sigma}^*,
    \end{align*}
    where in the first $\stackrel{d}{=}$ we used the induction hypothesis and the fact that the conditional distribution of $T_{\sigma}^*$ and $T_{\sigma}^{*,\Sigma}$ coincides on stopping before $t_1$, 
    and in the second $\stackrel{d}{=}$, we used Lemma \ref{lemma:in_block_exchange_invariance}. This concludes the proof.
\end{proof}

\begin{proof}[Proof of \Cref{thm:fast-fast_slow-slow_makes_no_difference}]
    Initialize Poisson clocks of appropriate parameters for every pair of fast and every pair of slow particles. The ringing times of these give rise to a scenario in the sense of \Cref{lemma:deterministic_scenario_independence}. Applying it, we see that the conditional distribution of $T_\sigma$ with respect to any scenario is identical to the distribution of $T_{\sigma}^*$. Hence by Fubini, $T_\sigma \stackrel{d}{=} T_{\sigma}^*$.
\end{proof}

\begin{proof}[Proof of \Cref{thm:sorting_time}]
Now we return to our problem with $K$ types as $N\to \infty$. According to \Cref{thm:fast-fast_slow-slow_makes_no_difference}, for fixed $N$ the $m$-sorting 
time $T_{N, m}$ starting from reverse-speed configuration can be modeled by a last passage time from 
$(1, 1)$ to $\left((K-m)n, mn\right)$. A subtle difficulty concerning the asymptotics of its last passage time is that different $N$s 
correspond to a different environment. However by our findings we can use a single environment to model $T_{N, m}$ for every $N$:
set
\[
    x_a=v_{m+((a-1)\bmod (K-m))},
    \qquad
    y_b=-v_{m-1-((b-1)\bmod m)}.
\]
In this single environment, for every $K, n$ we have 
$L^{\left((K-m)n, mn\right)} \stackrel{d}{=} T_{N, m}$.

Invoking \cite[Theorem~3.2]{EJS2021}, we find that
$$\frac{L^{\left((K-m)n, mn\right)}}{N}\xrightarrow{a.s.} \frac{1}{K}\inf_{v_{m-1}<z<v_{m}}\left(\sum_{i=0}^{m-1} \frac{1}{z-v_i}+\sum_{i=m}^{K-1} \frac{1}{v_i-z}\right).$$
As this result is proved via summable tail bounds, we obtain \eqref{eq:main1} by the distributional identity we proved. Taking maximum for $m$ yields \eqref{eq:main2}.
\end{proof}

\section*{Declaration of generative AI use}

During the preparation of this manuscript, the authors used ChatGPT (GPT-5.6) to assist in generating TikZ code for \Cref{fig:ferrers} according to author-provided specifications and in proofreading the manuscript. The authors reviewed and edited all generated material and take full responsibility for the accuracy and content of the manuscript.

\printbibliography

\end{document}